\documentclass[12pt,reqno]{amsart}
\usepackage{amsmath,amssymb}
\usepackage[colorlinks=true,linkcolor=magenta,citecolor=blue]{hyperref}
\usepackage[margin=1in]{geometry}
\usepackage{enumitem} 
\usepackage{xcolor}
\title[Quadratic systems in higher dimensions]{Uniform-in-time bounds for quadratic reaction-diffusion systems with mass dissipation in higher dimensions}
\author{Klemens Fellner, Jeff Morgan, Bao Quoc Tang}
\address{Klemens Fellner \hfill\break
	Institute of Mathematics and Scientific Computing, University of Graz, Heinrichstrasse 36, 8010 Graz, Austria}
\email{klemens.fellner@uni-graz.at}
\address{Jeff Morgan \hfill\break
Department of Mathematics, 
University of Houston, Houston, Texas 77004, USA}
\email{jmorgan@math.uh.edu}
\address{Bao Quoc Tang \hfill\break
	Institute of Mathematics and Scientific Computing, University of Graz, Heinrichstrasse 36, 8010 Graz, Austria}
\email{quoc.tang@uni-graz.at}

\renewcommand{\O}{\Omega}

\newcommand{\pa}{\partial}

\newcommand{\nx}{\nabla_{\!x}}
\newtheorem{theorem}{Theorem}[section]

\newtheorem{lemma}{Lemma}[section]
\newtheorem{proposition}{Proposition}[section]

\newtheorem{remark}{Remark}[section]

\begin{document}

	\begin{abstract}
		Uniform-in-time bounds of nonnegative classical solutions to reaction-diffusion systems in all space dimension are proved. The systems are assumed to dissipate the total mass and to have locally Lipschitz nonlinearities of at most (slightly super-) quadratic growth. This pushes forward the recent advances concerning global existence of reaction-diffusion systems dissipating mass in which a uniform-in-time bound has been known only in space dimension one or two. As an application, skew-symmetric Lotka-Volterra systems are shown to have unique classical solutions which are uniformly bounded in time in all dimensions with relatively compact trajectories in $C(\overline{\Omega})^m$.
	\end{abstract}

	\maketitle
	
	\tableofcontents
	
\section{Introduction and Main results}
Let $n\geq 1$ and $\Omega\subset\mathbb R^n$ be a bounded domain with smooth boundary $\partial\Omega$. We consider the following reaction-diffusion system
\begin{equation}\label{sys}
\begin{cases}
	\partial_t u_i - d_i\Delta u_i = f_i(u), &x\in\Omega, \; t>0,\; i=1,\ldots, m,\\
	\nx u_i \cdot \nu = 0, &x\in\partial\Omega, \; t>0,\; i=1,\ldots, m,\\
	u_i(x,0) = u_{i0}(x), &x\in\Omega, \; i=1,\ldots, m,
\end{cases}
\end{equation}
where $d_i$ are the diffusion coefficients, $\nu$ is the unit outward normal vector on $\partial\Omega$. The system \eqref{sys} is studied in this paper under the following assumptions on the data:
\begin{enumerate}
	[label = (D), ref = D]
	\item\label{D} $d_i>0$ for all $i=1,\ldots, m$,\\
	$u_{i0}\in L^\infty(\Omega)$ and $u_{i0}\geq 0$ a.e. in $\Omega$, for all $i=1,\ldots, m$,
\end{enumerate} and on the nonlinearities:
\begin{enumerate}
	[label=(A\theenumi),ref=A\theenumi]
	\item\label{A1} (Local Lipschitz and Quasi-Positivity) For all $i=1,\ldots, m$, $f_i:\mathbb R^m \to \mathbb R$ is locally Lipschitz continuous and
	\begin{equation}
		f_i(u) \geq 0 \quad \text{ for all } \quad u\in \mathbb R^m_+ \;\text{ with }\; u_i = 0.
	\end{equation}
	\item\label{A2} (Mass-Dissipation) For any $u\in \mathbb R^m_+$,
	\begin{equation*}
		\sum_{i=1}^mf_i(u) \leq 0.
	\end{equation*}
	\item\label{A3} (Slight Super-Quadratic Growth) There exists a constant $\varepsilon>0$ sufficiently small and a constant $K>0$ such that
	\begin{equation*}
	|f_i(u)| \leq K(1+|u|^{2+\varepsilon}) \quad \text{for all} \quad u\in \mathbb R^m,
	\end{equation*}
	for all $i=1, \ldots, m$.
\end{enumerate}

\medskip
Since the nonlinearities are locally Lipschitz, it is well known (see e.g. \cite{Ama85}) that under assumption \eqref{D}, there exists a unique local classical solution to \eqref{sys} on a maximal interval $(0,T_{\max})$\footnote{It is remarked that this local existence does not need the nonnegativity of the initial data.}. 
By a classical solution on a time interval $(0,T)$ we mean a vector of concentrations $u = (u_1, \ldots, u_m)$ which satisfies $u_i \in C([0,T);L^{p}(\O)) \cap C^{1,2}((0,T)\times \overline{\O})\cap L^\infty((0,T-\tau)\times\Omega)$ for all $n<p<\infty$, all $\tau\in (0,T)$, and $u$ satisfies each equation in \eqref{sys} pointwise.

Under the quasi-positivity assumption \eqref{A1}, this solution is nonnegative provided nonnegative initial data. Moreover, 
\begin{equation}\label{criteria}
\text{ if } \quad \lim_{t\uparrow T_{\max}}\|u_i(t)\|_{L^\infty(\Omega)} < +\infty,\; i=1,\ldots, m \quad \text{ then } \quad T_{\max} = +\infty.
\end{equation}
The global existence of this local solution is, on the other hand, rather delicate, except for the case $d_i = d$ for all $i=1,\ldots, m$. Indeed, summing up all equations yields
\begin{equation*}
	\partial_t\sum_{i=1}^m u_i - d\Delta \sum_{i=1}^m u_i \leq 0
\end{equation*}
and therefore by maximum principle
\begin{equation*}
	\biggl\|\sum_{i=1}^mu_i(t)\biggr\|_{L^\infty(\Omega)} \leq \biggl\|\sum_{i=1}^mu_{i0}\biggr\|_{L^\infty(\Omega)} \quad \text{ for all } \quad t\in (0,T_{\max}).
\end{equation*}
Considering nonnegative solutions, the above $L^{\infty}$ a priori estimate is sufficient in view of \eqref{criteria} to 
imply the global existence and uniform bounds of solutions. When the diffusion coefficients are different, the situation becomes suddenly much more difficult. The main reason is the lack of a maximum principle for general systems yielding  
$L^{\infty}$ a priori estimates. 

Under the assumptions of this paper, the mass dissipation condition \eqref{A2}, together with the homogeneous Neumann boundary, implies at least a uniform-in-time $L^1$ bound for all $t\in (0,T_{\max})$
\begin{equation}\label{sum_up}
	\frac{\partial}{\partial t}\sum_{i=1}^mu_i - \Delta\left(d_i\sum_{i=1}^m u_i\right) \leq 0 \quad \text{ thus } \quad \sum_{i=1}^m\int_{\Omega}u(x,t)dx \leq \sum_{i=1}^m\int_{\Omega}u_{i0}(x).
\end{equation}
Hence, from the nonnegativity of $u_i$, we have a uniform a priori estimate on solutions in $L^\infty(0,T;L^1(\Omega))$ as long as the solution exists. However, in view of \eqref{criteria} this is obviously not enough to obtain global existence of solutions to \eqref{sys}. 

Another set of available a priori estimates uses an elegant duality argument (see e.g. \cite{Mor89,PS97}). Such estimates show that $u_i\in L^2(\Omega\times(0,T))$, but even this is far from the criteria \eqref{criteria} to get $T_{\max} = +\infty$. In fact, it was shown in \cite{PS97} that there exist systems satisfying both \eqref{A1} and \eqref{A2} but have $L^{\infty}$ blow-up of solutions in finite time, i.e. $T_{\max}<+\infty$\footnote{It is noted that the counter-example was constructed for the case of inhomogeneous Dirichlet boundary conditions. Such a counter-example for homogeneous Neumann boundary conditions is still unknown, up to our knowledge.}. 

The counterexample of \cite{PS97} makes it clear that more conditions are required in order to obtain global existence. Of particular interest is the case when the nonlinearities are of (or are bounded by) polynomial type. Especially, the case of quadratic order has been extensively investigated as it models, for instance, binary reaction in chemistry like the following 
reversible reaction 
\begin{equation}\label{reversible_reaction}
	S_1 + S_2 \leftrightharpoons S_3 + S_4,
\end{equation}
which leads to a purely quadratic $4\times4$ reaction-diffusion system 
or general Lotka-Volterra systems in ecology.

\medskip
To put our work in context, let us review what has been done in the literature. Systems with conditions \eqref{A1}-\eqref{A2} have been an active research topic since the seventies as they appear frequently in applications. Let us mention an incomplete list \cite{CHS78,Rot84,Ama85,HMP87,Mor89,Mor90,FHM97} and many references therein. As mentioned above, with \eqref{A1}-\eqref{A2} alone, it is unlikely that one always has global existence of the classical solution. Therefore one either studies the global existence of a weaker notion of solution or imposes extra assumptions to get global classical solutions. Concerning the former direction, it was shown in \cite{Pie03} that if, for some reason, the nonlinearities are bounded in $L^1(\Omega\times(0,T))$, then one can prove the global existence of a weak solution. Thanks to various duality arguments, which give  $L^2(\Omega\times(0,T))$ or slightly better bounds on solutions, it follows that at most quadratic systems \eqref{sys}, in particular the prototypical system \eqref{reversible_reaction}, have global weak solutions \cite{DFPV07,CDF14}. We refer the interested reader to the extensive review \cite{Pie10} for more details.

More recently, Fischer showed in his paper \cite{Fis15} that if the nonlinearities satisfy an {\it entropy condition}, i.e.
\begin{equation}\label{entropy_condition}
\sum_{i=1}^m f_i(u)(\log u_i + \mu_i)\leq 0 \quad \text{ for all } \quad u \in (0,\infty)^m,
\end{equation}
for some $\mu_1, \ldots, \mu_m \in \mathbb R$, then there exists a global {\it renormalised solution} to \eqref{sys}. 
Note that renormalised solutions do not ensure $L^1$ integrability of the nonlinear reaction terms. 
However, renormalised solutions coincide with the unique classical solution as long as the latter exists \cite{Fis17}. 
Moreover, all renormalised solutions of complex balance reaction-diffusion systems (without boundary equilibria) 
converge exponentially to the positive equilibrium in the same way as weak or classical solutions, \cite{FT18}. 

About the latter research direction, i.e. global existence of classical solutions, the mass dissipation \eqref{A2} is often replaced 
 by the mass conservation condition
\begin{enumerate}
	[label=(A\theenumi'),ref=A\theenumi']
	\setcounter{enumi}{1}
	\item\label{A2'} (Mass-Conservation) 
	for any $u\in \mathbb R_+^m$, 
	\begin{equation*}
	\sum_{i=1}^mf_i(u) = 0.
	\end{equation*}
\end{enumerate}
Under \eqref{A1}-\eqref{A2'} and \eqref{entropy_condition}, it was obtained in \cite{GV10} that if $n=1$ and $f_i(u)$ are cubic, or if $n=2$ and $f_i(u)$ are quadratic, then the classical solution is global. This was reproved in \cite{Tan17a} by a different method, where also the $L^\infty$-norm of solutions was shown to grow at most polynomially in time. The global solutions to \eqref{sys} with quadratic nonlinearities in dimension $n\leq 2$ are in fact shown to be bounded uniformly in time \cite{PSY17}.
The dimension restriction was removed in \cite{CV09} for strictly sub-quadratic nonlinearities.
Without using the entropy condition \eqref{entropy_condition}, Desvillettes \textit{et al} showed by an improved duality method in \cite{CDF14} that if the diffusion coefficients are close to each other (depending on the dimension and the growth of the nonlinearities), then the classical solution exists globally with an $L^\infty$-norm growing at most polynomially. This polynomial growth of the $L^\infty$-norm is removed in the recent work \cite{CMT}. 

The global existence of classical solutions for \eqref{sys} with quadratic nonlinearities in higher dimensions, i.e. $n\geq 3$, had been open until  very recently when it was proved affirmatively\footnote{It is worth to mention that the case of whole space, i.e. $\Omega = \mathbb R^n$ and mass conservation condition \eqref{A2'} was in fact studied in an almost unnoticed paper by Kanel \cite{Kan90}.}. More precisely, in the whole space, i.e. $\Omega = \mathbb R^n$, \cite{CGV19} showed by the famous De Giorgi method under conditions \eqref{A1}-\eqref{A2'}-\eqref{A3} and \eqref{entropy_condition} that the strong solution exists globally but no a priori bound in time was given. Souplet \cite{Sou18} relaxed the mass conservation condition \eqref{A2'} to the mass dissipation condition \eqref{A2}, but still needed the entropy condition \eqref{entropy_condition}. It is remarked that this work treated both cases $\Omega = \mathbb R^n$ as well as $\Omega$ is a smooth bounded domain, and also gave no control in time of the solution. In our previous work \cite{FMT19} we were able to remove the entropy condition \eqref{entropy_condition} and thus proved the global existence of classical solutions by assuming only the mass dissipation condition \eqref{A2} (and of course conditions \eqref{A1} and \eqref{A3}), see Proposition~\ref{globalsol} below for the precise statement. 

We emphasise that most of these works do not give a control in time for the solution, except for \cite{FMT19} where the solution was shown to have $L^\infty$-norm growing at most polynomially in time. Such a polynomial growth is usually called {\it slowly growing a-priori estimate} and it becomes helpful when some exponential large time behaviour of the solution is known. For instance, for the reversible reaction \eqref{reversible_reaction}, we can interpolate this polynomial growth of the $L^\infty$-norm and the exponential convergence to equilibrium in $L^1$-norm (see e.g. \cite{CDF14,FT17}) to derive a uniform-in-time bound for the solution. Such a large time behaviour is however not always available, for instance considering the skew-symmetric Lotka-Volterra system (see Section \ref{application}). Therefore, it is interesting to ask: can we obtain a uniform-in-time bound for the solution to \eqref{sys} under conditions \eqref{A1}-\eqref{A2}-\eqref{A3} only? In this paper, we give an affirmative answer to this question. The main result of this paper is stated in the following theorem.
\begin{theorem}\label{thm:main}
	Assume that $\Omega\subset\mathbb R^n$, $n\geq 1$, is bounded with smooth boundary $\partial\Omega$. Let the assumptions \eqref{D} and \eqref{A1}-\eqref{A2}-\eqref{A3} hold. Then, system \eqref{sys} has a unique global classical nonnegative solution, which is bounded uniformly in time and space, i.e.
	\begin{equation*}
	\sup_{t>0}\sup_{i=1,\ldots, m}\|u_i(t)\|_{\infty,\Omega} < +\infty.
	\end{equation*}
\end{theorem}
\begin{remark}
	It is clear that the results of Theorem \ref{thm:main} still hold if \eqref{A2} is replaced by
	\begin{equation*}
		\sum_{i=1}^m\alpha_if_i(u) \leq 0
	\end{equation*}
	for some $(\alpha_i)_{i=1,\ldots, m}\in (0,\infty)^m$.
\end{remark}
\begin{remark}
	As mentioned above, the case of dimension $n\leq 2$ was proved already in e.g. \cite{PSY17}. Therefore, in the proof of Theorem \ref{thm:main}, we will only consider $n\geq 3$.
\end{remark}
\begin{remark}\label{equivalence}
One key observation is that the mass dissipation condition \eqref{A2} can be replaced by the mass conservation condition \eqref{A2'}. Indeed, by introducing a new unknown $u_{m+1}$ and inserting to \eqref{sys} an equation for $u_{m+1}$ as
\begin{equation*}
\partial_t u_{m+1} - \Delta u_{m+1} = -\sum_{i=1}^mf_i(u), \quad \nx u_{m+1}\cdot \nu = 0, \quad u_{m+1}(x,0) = 0
\end{equation*}
one obtains a new system of $(u_1, \ldots, u_{m}, u_{m+1})$, of which the nonlinearities satisfy a mass conservation condition. It's easy to see that the global existence and boundedness of solutions to this new system and \eqref{sys} are equivalent.
\end{remark}

Let us briefly describe the main ideas in the proof of Theorem \ref{thm:main}. Because of Remark \ref{equivalence}, we will consider for the rest of this paper the mass conservation condition \eqref{A2'}. The global existence of a classical solution was shown in \cite{FMT19}.  To obtain the uniform-in-time bound for this solution, we will show that
\begin{equation}\label{uniform}
	\sup_{s\in [\tau,\tau+1]}\|u_i(s)\|_{L^\infty(\Omega)} \leq C \quad \text{ for all } i = 1, \ldots, m \quad \text{ and all} \quad \tau\in \mathbb R,
\end{equation}
where {\it $C$ is a constant independent of $\tau$}. Obviously, this implies the desired uniform bound. In order to prove \eqref{uniform} we first use a duality argument to get
\begin{equation*}
	\int_{\tau}^{\tau+1}\!\!\int_{\Omega}|u_i(s,x)|^2dxds \leq C \quad \text{ for all } i=1,\ldots, m, \; \tau \in \mathbb R,
\end{equation*}
with $C$ is independent of $\tau$. A key idea of this paper is then a specific bootstrap argument leading to the crucial estimate
\begin{equation}\label{uniform_1}
\left\|\int_{\tau}^{\tau+1}u_i(s,x)ds\right\|_{L^\infty(\Omega)} \leq C \quad \text{ for all } i=1,\ldots, m, \; \tau \in \mathbb R,
\end{equation}
for some $C$ is independent of $\tau$. Having shown \eqref{uniform_1}, we can reuse ideas from \cite{FMT19} on each internal $[\tau,\tau+1]$ to obtain 
\begin{equation*}
	\sup_{(x, t)\in \Omega\times [\tau,\tau+1]}|u_i(x,t)| \leq C \quad \text{ for all } i=1,\ldots, m, \; \tau \in \mathbb R
\end{equation*}
where $C$ is again independent of $\tau$, and thus \eqref{uniform}.

\medskip
\textbf{The paper is organised as follows}: In the next section, we give some preliminary results for \eqref{sys}. The proof of Theorem \ref{thm:main} is presented in Section \ref{sec:proof_main}. Finally, an application of the main results to skew-symmetric Lotka-Volterra is shown in Section \ref{application}.

\medskip
\noindent{\bf Notation.}
We will use the following notation.
\begin{itemize}
	\item For $1\leq p\leq \infty$ and $\tau <T$, the usual norms of $L^p(\Omega)$ and $L^p(\Omega\times (\tau,T))$ are denoted by
	\begin{equation*}
		\|u\|_{p,\Omega} = \|u\|_{L^p(\Omega)} \quad \text{ and } \quad \|u\|_{p,\Omega\times(\tau,T)} = \|u\|_{L^p(\Omega\times(\tau, T))},
	\end{equation*}
	respectively.
	
	\item $W^{k,p}(\Omega)$ denotes the classical Sobolev space with $k\in \mathbb N$ and $1\leq p \leq \infty$, and its norm is written as $\|\cdot\|_{W^{k,p}(\Omega)}$.
	
	\item We denote for $r \in \mathbb N$ and $\alpha\in \mathbb N^n$
	\begin{equation*}
		W^{2,1}_{\Omega\times (\tau, T)} = \{f: \Omega\times(\tau, T)\to \mathbb R\,|\; \|\partial_t^r\partial_x^\alpha f\|_{2,\Omega\times(\tau,T)} <+\infty \text{ for all } 2r + |\alpha| \leq 2\}
	\end{equation*}
	associated with the norm
	\begin{equation*}
		\|f\|^{(2,1)}_{\Omega\times(\tau,T)} = \sum_{2r+|\alpha| \leq 2}\|\partial_t^r\partial_x^\alpha f\|_{2,\Omega\times(\tau,T)}.
	\end{equation*}
	
	\item We write $C$ for some generic constant which can be different from line to line, or even in the same line, but {\it $C$ is always independent of $\tau$} (unless explicitly stated otherwise).
\end{itemize}

\section{Preliminaries}
We recall the global existence of the classical solution, which was proved in \cite{FMT19}. 
\begin{proposition}\cite{FMT19}\label{globalsol}
	Let $\Omega\subset\mathbb R^n$, $n\geq 1$ be bounded with smooth boundary $\partial\Omega$. Assume conditions \eqref{D},  and \eqref{A1}-\eqref{A2}-\eqref{A3}. Then, there exists a unique global, nonnegative, classical solution to \eqref{sys}, and the $L^\infty$-norm grows at most polynomially in time, i.e.
	\begin{equation*}
	\sup_{i=1,\ldots, m}\|u_i(t)\|_{\infty,\Omega} \leq C(1+t^\xi) \quad \text{ for all } \quad t>0,
	\end{equation*}
	for some $\xi >0$, $C>0$ are independent of $t$. 
\end{proposition}
\begin{lemma}[Regularity Interpolation]\cite[Lemma 1.1]{FMT19} \label{key_lem}\hfill\\
	For some constant $d>0$, let $u$ satisfy the inhomogeneous linear heat equation
	\begin{equation}\label{heat-eq}
	\begin{cases}
	u_t - d\Delta u = \psi(x,t), &\quad(x,t)\in \Omega\times(\tau,T),\\
	\nx u \cdot \nu = 0, &\quad(x,t)\in \pa\O\times (\tau,T)\\
	u(x,\tau) = u_0(x), &\quad x\in\O
	\end{cases}
	\end{equation}
	where $\tau < T$. Assume that
	\begin{itemize}[topsep=5pt, leftmargin=10mm]
		\item[(i)] there exists a H\"older exponent $\gamma \in [0,1)$ such that for all $x, x' \in \Omega$, and all $t\in (\tau,T)$, 
		\begin{equation}\label{AA1}
		|u(x,t) - u(x',t)| \leq H|x - x'|^{\gamma},
		\end{equation}
		\item[(ii)] the inhomogeneity satisfies 
		\begin{equation}\label{AA2}
		\sup_{\Omega\times(\tau,T)}|\psi(x,t)| \leq F.
		\end{equation}
	\end{itemize}
	Then, the following uniform gradient estimate follows:
	\begin{equation*}
	|\nx u(x,t)| \leq C\|u_0\|_{C^1(\Omega)} + BH^{\frac{1}{2-\gamma}}F^{\frac{1-\gamma}{2-\gamma}}, \qquad \text{for all }\quad (x,t)\in \Omega\times(\tau,T),
	\end{equation*}
	where $B>0$ and $C>0$ are constants depending only on $\O$, $n$, $d$ and $\gamma$.
\end{lemma}

\begin{lemma}\label{key_lem1}
	Let $u$ satisfy \eqref{heat-eq} with zero initial data $u(x,\tau) = 0$. Then there exists a constant $C$ independent of $T$ and $\tau$ such that
	\begin{equation*}
		\sup_{\Omega\times(\tau,T)}|\Delta u(x,t)| \leq C \sup_{\Omega\times(\tau,T)}(|u(x,t)| + |\nx u(x,t)|)^{1/2} \sup_{\Omega\times(\tau,T)}(|\psi(x,t)| + |\nx \psi(x,t)|)^{1/2},
	\end{equation*}
(whenever the right hand side is finite). 
\end{lemma}
\begin{proof}
	See \cite[Lemma 3.9]{FMT19}.
\end{proof}

\begin{lemma}[Maximal regularity of heat equation]\cite{Lam87}\label{max-reg}
	For $0<\tau<T$, let $u$ be the solution to 
	\begin{equation}\label{backward}
		\begin{cases}
		\partial_t \phi + d\Delta \phi = -\theta, &x\in\Omega, t\in (\tau,T),\\
		\nx \phi \cdot \nu = 0, &x\in\partial\Omega, t\in (\tau,T),\\
		\phi(x,T) = 0, &x\in\Omega,
		\end{cases}
	\end{equation}
	where $0\leq \theta \in L^2(\Omega\times(\tau,T))$. Then we have
	\begin{equation}\label{e0}
		\|\phi\|^{(2,1)}_{{\Omega\times (\tau,T)}} \leq C(\Omega,T-\tau)\|\theta\|_{2,\Omega\times(\tau,T)}
	\end{equation}
	where the constant $C(\Omega,T-\tau)$ depends on $T-\tau$ but not on $T$ and $\tau$ specifically,
	and
	\begin{equation}\label{e1}
		\|\Delta \phi\|_{2,\Omega\times(\tau,T)} \leq \frac{1}{d}\|\theta\|_{2,\Omega\times(\tau,T)}.
	\end{equation}
\end{lemma}
\begin{proof} At first glance, \eqref{backward} looks like a backward heat equation. However, with the transformation 
$\Phi(T-t,x)=\phi(t,x)$, it becomes the classical heat equation with homogeneous initial data, for which
the maximal regularity \eqref{e0} is well-known, see e.g. \cite{Lam87}. Here we prove \eqref{e1} for the sake of completeness.
	By multiplying the equation of $\phi$ with $\Delta \phi$ and integrating on $\Omega\times(\tau,T)$ we have
	\begin{equation*}
		\|\nx \phi(0)\|_{2,\Omega}^2 + d\int_\tau^T\int_{\Omega}|\Delta \phi|^2dxdt = -\int_\tau^T\int_{\Omega}\theta \Delta \phi dxdt \leq \|\theta\|_{2,\Omega\times(\tau,T)}\|\Delta\phi\|_{2,\Omega\times(\tau,T)},
	\end{equation*}
	hence the desired estimate \eqref{e1}.
\end{proof}
\begin{lemma}[Embeddings]\label{embedding}\cite{LSU88}
	There exists a constant $C(\Omega,T-\tau)$ depending on $\Omega$ and $T-\tau$ such that
	\begin{equation*}
		\|u\|_{q,\Omega\times(\tau,T)} \leq C(q,\Omega,T-\tau)\|u\|^{(2,1)}_{\Omega\times(\tau,T)} \quad \text{ for all } \quad u \in W^{2,1}_{\Omega\times(\tau,T)}
	\end{equation*}
	 for all $q<\frac{2(n+2)}{n-2}$. In particular, by choosing $q = \frac{2(n+1)}{n-2}$ we have
	\begin{equation*}
		\|u\|_{\frac{2(n+1)}{n-2},\Omega\times(\tau,T)} \leq C(n,\Omega,T-\tau)\|u\|^{(2,1)}_{\Omega\times(\tau,T)}.
	\end{equation*}
\end{lemma}
\begin{lemma}\label{elementary}
	Let $\{a_n\}_{n\geq 0}$ be a sequence of nonnegative real numbers. Denote by $\mathfrak N = \{k \in \mathbb N: a_{k}\leq a_{k+1} \}$. If there exists a constant $C_0$ independent of all $k\in \mathfrak N$ such that
	\begin{equation*}
		a_k \leq C_0 \quad \text{ for all } \quad k \in \mathfrak N,
	\end{equation*}
	then
	\begin{equation*}
		a_k \leq C_1=\max\{a_0,C_0\} \quad \text{ for all } \quad k \in \mathbb N
	\end{equation*}
	for a constant $C_1$ independent of $k$.
\end{lemma}
\begin{proof}
	The proof of this lemma is straightforward.
\end{proof}
	
\section{Proof of Theorem \ref{thm:main}}\label{sec:proof_main}
From the mass dissipation condition \eqref{A2} or the mass conservation condition \eqref{A2'} as well as the nonnegativity of the solution, we have the following uniform-in-time bound of the $L^1$-norm.
\begin{lemma}\label{lem:L1bound}
	There exists a constant $M>0$ such that
	\begin{equation*}
		\sum_{i=1}^{m}\|u_i(t)\|_{1,\Omega} \leq M
	\end{equation*}
	for all $t\geq 0$.
\end{lemma}
\begin{proof}
	Summing up the equations in \eqref{sys} and integrating over $\Omega$, we get, thanks to \eqref{A2} and the homogeneous Neumann boundary condition
	\begin{equation*}
		\partial_t\sum_{i=1}^m\int_{\Omega}u_i(x,t)dx = 0,
	\end{equation*}
	hence the desired estimates.
\end{proof}

For the rest of the paper, thanks to Remark \ref{equivalence}, we will work (w.l.o.g.) with the mass conservation condition \eqref{A2'} instead of \eqref{A2}.
For later use, we denote by
\begin{equation}\label{maxmin}
	d_{\max} = \max_{i=1,\ldots, m}d_i \quad \text{ and } \quad d_{\min} = \min_{i=1,\ldots, m}d_i.
\end{equation}
\begin{lemma}[$L^2$-duality estimate]\label{lem1}
	For $n\ge3$, there exists a constant $C$ independent of $\tau$ such that for all $i=1,\ldots, m$
	\begin{equation*}
		\|u_i\|_{2,\Omega\times(\tau,\tau+1)} \leq 
		M^{\frac{n-2}{n+1}} \left(\sqrt{2}\,\frac{d_{\max}}{d_{\min}}\right)^{\frac{n-2}{3}}
	\end{equation*}
	for all $\tau\geq 0$.
\end{lemma}
\begin{proof}
	We use the duality method. Let $0 \leq \theta \in L^2(\Omega\times(\tau,\tau+2))$ and $\phi$ be the solution to 
	\begin{equation}\label{e1_1}
	\begin{cases}
		\partial_t \phi + d_{\max}\Delta \phi = -\theta, &x\in\Omega, \; t\in (\tau,\tau+2),\\
		\nx \phi \cdot \nu = 0, &x\in\partial\Omega, \; t>0,\\
		\phi(x,\tau+2) = 0, &x\in\Omega,
	\end{cases}
	\end{equation}
	By the maximal regularity in Lemma \ref{max-reg} we have in particular
	\begin{equation*}
		\|\phi\|_{W^{2,1}_{\Omega\times(\tau,\tau+2)}} \leq C_{\Omega}\|\theta\|_{2,\Omega\times(\tau,\tau+2)},
	\end{equation*}
	and thus, by the embedding in Lemma \ref{embedding} with $q = \frac{2(n+1)}{n-2}$,
	\begin{equation}\label{bound_phi}
		\|\phi\|_{q,\Omega\times(\tau,\tau+2)} \leq C\|\theta\|_{2,\Omega\times(\tau,\tau+2)},
	\end{equation}
	for a constant $C$ independent of $\tau$.

	Let $\varphi: \mathbb R \to [0,1]$ be a smooth function such that $\varphi(s) = 0$ for $s\leq 0$, $\varphi(s) = 1$ for $s\geq 1$ and $0 \leq \varphi'(s) \leq 2$ for all $s\in \mathbb R$. Define the shifted function $\varphi_\tau(\cdot) = \varphi(\cdot - \tau)$. 
	By multiplying the equation of $u_i$ by $\varphi_\tau$ we have
	\begin{equation*}
		\partial_t(\varphi_\tau u_i) - d_i\Delta (\varphi_\tau u_i) = \varphi_\tau' u_i + \varphi_\tau f_i(u).
	\end{equation*}
	By testing with solutions of \eqref{e1_1} and integration by parts we calculate 
	\begin{equation}\label{e2}
	\begin{aligned}
	&\int_{\tau}^{\tau+2}\!\!\int_{\Omega}(\varphi_\tau u_i)\theta dxdt\\
	&= \int_{\tau}^{\tau+2}\!\!\int_{\Omega}(\varphi_\tau u_i) (-\partial_t \phi - d_{\max}\Delta \phi)dxdt\\
	&= \int_{\tau}^{\tau+2}\!\!\int_{\Omega}\phi(\partial_t (\varphi_\tau u_i) - d_i\Delta (\varphi_\tau u_i))dxdt + (d_i-d_{\max})\int_{\tau}^{\tau+2}\!\!\int_{\Omega}(\varphi_\tau u_i)\Delta \phi dxdt  \\
	& =(d_i-d_{\max})\int_{\tau}^{\tau+2}\!\!\int_{\Omega}(\varphi_\tau u_i)\Delta \phi dxdt + \int_{\tau}^{\tau+2}\!\!\int_{\Omega}\phi \left(\varphi_\tau'u_i  +\varphi_\tau f_i(u)\right)dxdt.
	\end{aligned}
	\end{equation}
	Summing up this equality for $i=1,\ldots, m$ and using the mass conservation condition \eqref{A2'}, we have
	with $\varphi'_\tau\le 2$
	\begin{equation}\label{e3}
	\begin{aligned}
	&\int_{\tau}^{\tau+2}\!\!\int_{\Omega}\sum_{i=1}^m(\varphi_\tau u_i) \theta dxdt\\
	&\leq(d_{\max} - d_{\min})\int_{\tau}^{\tau+2}\!\!\int_{\Omega}|\Delta \phi|\sum_{i=1}^{m}(\varphi_\tau u_i)dxdt + 2\int_{\tau}^{\tau+2}\!\!\int_{\Omega}\phi \sum_{i=1}^mu_i dxdt\\
	&\leq  (d_{\max}-d_{\min})\|\Delta\phi\|_{2,\Omega\times(\tau,\tau+2)}\biggl\|\sum_{i=1}^m\varphi_\tau u_i \biggr\|_{2,\Omega\times(\tau,\tau+2)} + 2\int_{\tau}^{\tau+2}\!\!\int_{\Omega}\phi \sum_{i=1}^mu_i dxdt.
	\end{aligned}
	\end{equation}
	For $n\ge3$, we use H\"older's inequality with $q = \frac{2(n+1)}{n-2}$ and $q' = \frac{q}{q-1} = \frac{2n+2}{n+4}$ 
	and \eqref{bound_phi}
	\begin{equation}\label{f1}
	\int_{\tau}^{\tau+2}\!\!\!\int_{\Omega}\phi \sum_{i=1}^mu_i dxdt \leq \|\phi\|_{q,\Omega\times(\tau,\tau+2)}\Bigl\|\sum_{i=1}^mu_i\Bigr\|_{q',\Omega\times(\tau,\tau+2)} \leq C\|\theta\|_{2,\Omega\times(\tau,\tau+2)}\Bigl\|\sum_{i=1}^mu_i\Bigr\|_{q',\Omega\times(\tau,\tau+2)}.
	\end{equation}
	By abbreviating $z := \sum_{i=1}^mu_i$, we apply H\"older once more to estimate 
	\begin{equation}\label{f2}
	\begin{aligned}
		\|z\|_{q',\Omega\times(\tau,\tau+2)} &=  \left(\int_{\tau}^{\tau+2}\!\!\int_{\Omega}z^{\frac{2n+2}{n+4}}dxdt \right)^{\frac{n+4}{2n+2}}
		= \left(\int_{\tau}^{\tau+2}\!\!\int_{\Omega}z^{\frac{2n-4}{n+4}}z^{\frac{6}{n+4}}dxdt \right)^{\frac{n+4}{2n+2}}\\
		&\leq \left[\left(\int_\tau^{\tau+2}\!\!\int_{\Omega}(z^{\frac{2n-4}{n+4}})^{\frac{n+4}{n-2}}dxdt\right)^{\frac{n-2}{n+4}}\left(\int_{\tau}^{\tau+2}\!\!\int_{\Omega}zdxdt\right)^{\frac{6}{n+4}} \right]^{\frac{n+4}{2n+2}}\\
		&\leq C(M)\|z\|_{2,\Omega\times(\tau,\tau+2)}^{\frac{n-2}{n+1}},
	\end{aligned}
	\end{equation}
	where $C(M)=M^{3/(n+1)}$ and $M$ is given by Lemma \ref{lem:L1bound}.
	By inserting \eqref{f1} and \eqref{f2} into \eqref{e3}, and applying Lemma \ref{max-reg}, we obtain
	\begin{equation}\label{f3}
	\begin{aligned}
		\int_{\tau}^{\tau+2}\!\!\int_{\Omega}\sum_{i=1}^m(\varphi_\tau u_i)\theta dxdt &\leq \frac{d_{\max} - d_{\min}}{d_{\max}}\|\theta\|_{2,\Omega\times(\tau,\tau+2)}\biggl\|\sum_{i=1}^m\varphi_\tau u_i\biggr\|_{2,\Omega\times(\tau,\tau+2)}\\
		&\quad + C(M)\|\theta\|_{2,\Omega\times(\tau,\tau+2)}\biggl\|\sum_{i=1}^mu_i\biggr\|_{2,\Omega\times(\tau,\tau+2)}^{\frac{n-2}{n+1}}.
	\end{aligned}
	\end{equation}
	Since this holds for all $\theta\ge0$, duality implies
	\begin{equation*}
		\biggl\|\sum_{i=1}^m\varphi_\tau u_i\biggr\|_{2,\Omega\times(\tau,\tau+2)} \leq \frac{d_{\max}-d_{\min}}{d_{\max}}\biggl\|\sum_{i=1}^m\varphi_\tau u_i\biggr\|_{2,\Omega\times(\tau,\tau+2)} + C(M)\biggl\|\sum_{i=1}^m u_i\biggr\|_{2,\Omega\times(\tau,\tau+2)}^{\frac{n-2}{n+1}}
	\end{equation*}
	and consequently, by recalling $\varphi_\tau(s) = 1$ for $s\in [\tau+1,\tau+2]$ 
	\begin{equation}\label{f4}
		\biggl\|\sum_{i=1}^m u_i\biggr\|_{2,\Omega\times(\tau+1,\tau+2)} \le \biggl\|\sum_{i=1}^m\varphi_\tau u_i\biggr\|_{2,\Omega\times(\tau,\tau+2)} \leq \frac{C(M)d_{\max}}{d_{\min}}\biggl\|\sum_{i=1}^m u_i\biggr\|_{2,\Omega\times(\tau,\tau+2)}^{\frac{n-2}{n+1}}.
	\end{equation}
	We now suppose $\tau\in \mathbb N$ and consider $\tau$ such that 
	\begin{equation}\label{increase_tau1}
		\biggl\|\sum_{i=1}^mu_i\biggr\|_{2,\Omega\times(\tau,\tau+1)} \leq \biggl\|\sum_{i=1}^mu_i\biggr\|_{2,\Omega\times(\tau+1,\tau+2)},
	\end{equation}
	which implies 
	\begin{equation*}
	\biggl\|\sum_{i=1}^mu_i\biggr\|_{2,\Omega\times(\tau,\tau+2)} \leq \sqrt{2} \biggl\|\sum_{i=1}^mu_i\biggr\|_{2,\Omega\times(\tau+1,\tau+2)}.
	\end{equation*}
	It follows from \eqref{f4} that
	\begin{equation*}
		\biggl\|\sum_{i=1}^m u_i\biggr\|_{2,\Omega\times(\tau+1,\tau+2)} \leq \frac{2^{\frac{n-2}{2(n+1)}}C(M)d_{\max}}{d_{\min}} \biggl\|\sum_{i=1}^m u_i\biggr\|_{2,\Omega\times(\tau+1,\tau+2)}^{\frac{n-2}{n+1}}.
	\end{equation*}
	Finally by Young's inequality we get for all $\tau\in\mathbb N$ such that condition \eqref{increase_tau1} holds
	\begin{equation*}
		\biggl\|\sum_{i=1}^m u_i\biggr\|_{2,\Omega\times(\tau+1,\tau+2)} \leq 2^{\frac{n-2}{6}} \left(\frac{C(M)d_{\max}}{d_{\min}}\right)^{\frac{n-2}{3}}
	\end{equation*}
	Thanks to Lemma \ref{elementary}, this is also true for all $\tau\in \mathbb N$, and thus the proof of Lemma \ref{lem1} is finished.
\end{proof}

\begin{lemma}\label{lem2}
	There exists a constant $C$ independent of $\tau$ such that for all $i=1,\ldots, m$
	\begin{equation*}
		\left\|\int_{\tau}^{\tau+1}u_i(x,t)dt\right\|_{\infty,\Omega} \leq C
	\end{equation*}
	for all $\tau\geq 0$.
\end{lemma}
\begin{proof}
	 Define $w_i(x,t) = \varphi_\tau u_i(x,t)$, where $\varphi_\tau$ is the shifted cut-off function defined in Lemma \ref{lem1}. We have
	\begin{equation}\label{e8}
		\partial_t w_i - d_i\Delta w_i = \varphi'_\tau u_i + \varphi_\tau f_i(u), \quad \nx w_i\cdot \nu = 0, \quad w_i(x,\tau) = 0.
	\end{equation}
	Integrating \eqref{e8} on $(\tau,\tau+2)$ and summing over $i=1,\ldots, m$, we get
	\begin{equation*}
		\sum_{i=1}^mw_i(x,\tau+2) - \Delta\left(\sum_{i=1}^m\int_{\tau}^{\tau+2}d_iw_i(x,t)dt\right) = \sum_{i=1}^{m}\int_\tau^{\tau+2}\varphi'_\tau u_i(x,t)dt
	\end{equation*}
	 thanks to the condition \eqref{A2'}. Hence
	 \begin{equation}\label{e9}
		 -\Delta\left(\sum_{i=1}^m\int_{\tau}^{\tau+2}d_iw_i(x,t)dt\right) + \left(\sum_{i=1}^m\int_{\tau}^{\tau+2}d_iw_i(x,t)dt\right) \leq H_\tau(x)
	 \end{equation}
	 subject to homogeneous Neumann boundary condition, with 
	 \begin{equation}\label{def_H}
	 	H_\tau(x) = \left(\sum_{i=1}^m\int_{\tau}^{\tau+2}d_iw_i(x,t)dt\right) + \sum_{i=1}^m\int_{\tau}^{\tau+2}\varphi'_\tau u_i(x,t)dt.
	 \end{equation}
	 Let $\Psi$ be the solution to
	 \begin{equation}\label{eq_Psi}
		 -\Delta \Psi + \Psi = H_\tau(x), \; x\in\Omega, \qquad \nx \Psi \cdot \nu = 0, \; x\in\partial\Omega.
	 \end{equation}
	 By comparison we have
	 \begin{equation}\label{comparison}
		 \sum_{i=1}^m\int_{\tau}^{\tau+2}d_iw_i(x,t)dt \leq \Psi(x) \quad \text{ for a.e. } \quad x\in\Omega.
	 \end{equation}
	 For any $1<p<\infty$, the maximal regularity for linear elliptic equation gives
	 \begin{equation}\label{elliptic}
		 \|\Psi\|_{W^{2,p}(\Omega)} \leq C(p)\|H_\tau\|_{\Omega,p}.
	 \end{equation}
	 By Sobolev's embedding we have
	 \begin{equation}\label{f4_1}
		 \|\Psi\|_{q,\Omega} \leq C(p)\|H_\tau\|_{p,\Omega}
	 \end{equation}
	 for $q = \infty$ if $p > \frac{n}{2}$ and $q<\frac{np}{n-2p}$ arbitrary if $p\leq \frac{n}{2}$. In the later case, by choosing $q = \frac{np}{n-2}$, 
	 \begin{equation*}
		 \|\Psi\|_{\frac{np}{n-2},\Omega} \leq C\|H_\tau\|_{p,\Omega}.
	 \end{equation*}	
	 By \eqref{comparison} it follows that
	 \begin{equation}\label{f5}
	 	\biggl\|\sum_{i=1}^m\int_{\tau}^{\tau+2}d_iw_i(x,t)dt \biggr\|_{\frac{np}{n-2},\Omega} \leq C(p)\|H_\tau\|_{p,\Omega}.
	 \end{equation}
	 
	 The bootstrap procedure goes as follows. Firstly, from Lemma \ref{lem1} and the fact that $0\leq \varphi, \varphi' \leq 2$, it yields
	 \begin{equation*}
		 \|H_\tau\|_{2,\Omega} \leq C
	 \end{equation*}
	 with $C$ is independent of $\tau$. 
	 If $2>\frac{n}{2}$, i.e. $n=3$, we get from the Sobolev embedding \eqref{f4_1}
	 \begin{equation*}
		 \biggl\|\sum_{i=1}^m\int_{\tau}^{\tau+2}d_iw_i(x,t)dt \biggr\|_{\infty,\Omega} \leq C\|H_\tau\|_{2,\Omega} \leq C.
	 \end{equation*}
	Otherwise, for $n\ge4$, \eqref{f5} gives
	\begin{equation}\label{f6}
		\biggl\|\sum_{i=1}^m\int_{\tau}^{\tau+2}d_iw_i(x,t)dt \biggr\|_{\frac{2n}{n-2},\Omega} \leq C\|H_\tau\|_{2,\Omega} \leq C.
	\end{equation}
	Recalling $w_i(x,t) = \varphi_\tau u_i(x,t)$, and thus $w_i(x,t) = u_i(x,t)$ for all $t\in [\tau+1,\tau+2]$, we obtain from \eqref{f6} that
	\begin{equation}\label{f7}
		\biggl\|\sum_{i=1}^m\int_{\tau+1}^{\tau+2}u_i(x,t)dt\biggr\|_{\frac{2n}{n-2},\Omega} \leq C, \qquad \text{for all}\quad \tau >0.
	\end{equation}
	We get from \eqref{f7} and the definition of $H_\tau$ in \eqref{def_H} that
	\begin{equation}\label{f8}
		\|H_\tau\|_{\frac{2n}{n-2},\Omega} \leq C
	\end{equation}
	with $C$ independent of $\tau$, for all $\tau>0$. 
	
	Now we can repeat the bootstrap argument to get 
	\begin{equation*}
		\|H_\tau\|_{2\left(\frac{n}{n-2}\right)^k, \Omega} \leq C
	\end{equation*}
	for all $k\geq 1$ and $C$ independent of $\tau$. Choose $k$ large enough such that $2\left(\frac{n}{n-2}\right)^k > \frac{n}{2}$, we obtain after a final Sobolev embedding \eqref{f4_1}
	\begin{equation*}
		\biggl\|\sum_{i=1}^m\int_{\tau}^{\tau+2}d_iw_i(x,t)dt \biggr\|_{\infty,\Omega} \leq C
	\end{equation*}
	and consequently
	\begin{equation*}
		\biggl\|\int_{\tau+1}^{\tau+2}u_i(x,t)dt \biggr\|_{\infty,\Omega} \leq C
	\end{equation*}
	for all $i=1,\ldots, m$, all $\tau>0$, and $C$ independent of $\tau$. This completes the proof of Lemma \ref{lem2}.
\end{proof}

Summing up the equation \eqref{e8} for $i=1,\ldots, m$, and using condition \eqref{A2'} lead to
\begin{equation}\label{e11}
	\partial_t\left(\sum_{i=1}^m w_i \right) - \Delta\left(\sum_{i=1}^md_iw_i \right) = \varphi'_\tau\sum_{i=1}^mu_i.
\end{equation}
We set 
\begin{equation}\label{def_v}
v(x,t) = \int_\tau^{t}\sum_{i=1}^md_iw_i(x,s)ds \quad \text{ for } \quad t\in (\tau,\tau+2).
\end{equation}
Integrating \eqref{e11} on $(\tau,\tau+2)$ and recalling that $w_i(x,\tau) = \varphi(0)u_i(x,\tau) = 0$, we get
\begin{equation}\label{eq_v}
	\begin{cases}
		b(x,t)\partial_tv(x,t) - \Delta v(x,t) = \varphi'_\tau \int_{\tau}^{t} \sum_{i=1}^m u_i(x,s)ds, &x\in\Omega, \; t\in (\tau,\tau+2),\\
		\nx v(x,t)\cdot \nu = 0, & x\in\partial\Omega,\; t\in (\tau,\tau+2)\\
		v(x,\tau) = 0, &x\in\Omega,
	\end{cases}
\end{equation}
where
\begin{equation}\label{def_b}	
	\frac{1}{d_{\max}} \leq b(x,t):= \frac{\sum_{i=1}^mw_i(x,t)}{\sum_{i=1}^md_iw_i(x,t)} \leq \frac{1}{d_{\min}}.
\end{equation}
\begin{lemma}\label{lem3}
	There exist constants $C>0$ and $\delta \in (0,1)$ {\normalfont independent of $\tau$} such that
	\begin{equation}\label{Linf_v}
		\|v(x,t)\|_{\infty,\Omega\times(\tau,\tau+2)} \leq C
	\end{equation}
	and
	\begin{equation}\label{Holder_v}
		|v(x,t) - v(x',t)| \leq C|x - x'|^{\delta} \quad \text{ for all } \quad x, x'\in\Omega, \; t\in (\tau,\tau+2).
	\end{equation}
\end{lemma}
\begin{proof}
	The uniform bounds of $v$ independent of $\tau$ follows directly from Lemma \ref{lem2}. The H\"older continuity of $v$ was proved in \cite[Appendix A]{FMT19} with the remark that the constant is $C$ in \eqref{Holder_v} is independent of $\tau$ thanks to the $L^\infty$-bound of $v$ in \eqref{Linf_v}.
\end{proof}

\begin{lemma}\label{lem3_1}
	There exists a constant $C$ independent of $\tau$ such that 
	\begin{equation}\label{e12}
		\sup_{\Omega\times(\tau,\tau+2)}|\nx v(x,t)| \leq C(1+U^{\frac{1-\delta}{2-\delta}})
	\end{equation}
	where 
	\begin{equation*}
		U = \sup_{\Omega\times(\tau,\tau+2)}\sup_{i=1,\ldots, m}|u_i(x,t)|.
	\end{equation*}
\end{lemma}
\begin{proof}
From \eqref{e11} it follows that, $v$ is also the solution to 
\begin{equation}\label{eq_v1}
\begin{cases}
	\partial_t v(x,t) - \Delta v(x,t) = \psi(x,t), &x\in\Omega,\; t\in (\tau,\tau+2),\\
	\nx v(x,t)\cdot\nu = 0, &x\in\partial\Omega,\; t\in (\tau,\tau+2),\\
	v(x,\tau) = 0, &x\in\Omega
\end{cases}
\end{equation}
with
\begin{equation}\label{psi}
	\psi(x,t) = \varphi'_\tau \int_{\tau}^{t} \sum_{i=1}^m u_i(x,s)ds + \sum_{i=1}^m(d_i - 1)w_i(x,t).
\end{equation}
By applying Lemma \ref{key_lem}, we have
\begin{equation}
	\sup_{\Omega\times(\tau,\tau+2)}|\nx v(x,t)| \leq C\sup_{\Omega\times(\tau,\tau+2)}|\psi(x,t)|^{\frac{1-\delta}{2-\delta}}
\end{equation}
where $\delta$ is in Lemma \ref{lem3}. Using the uniform bounds in Lemma \ref{lem2} and $w_i(x,t) = \varphi_\tau u_i(x,t)$ we obtain the desired estimate \eqref{e12}.
\end{proof}
To estimate $\Delta v$ we need the following two lemmas.
\begin{lemma}\label{lem4}
	Let $w_i$ be the solution to \eqref{e8}, we have
	\begin{equation*}
		\sup_{\Omega\times(\tau,\tau+2)}|\nx w_i(x,t)| \leq C\left(1 + U^{\frac{3+\varepsilon}{2}} \right)
	\end{equation*}
	where the constant $C$ is independent of $\tau$.
\end{lemma}
\begin{proof}
	By applying Lemma \ref{key_lem} with $\gamma=0$ to equation \eqref{e8} and using the assumption \eqref{A3} we get
	with $W = \sup_{\Omega\times(\tau,\tau+2)}\sup_{i=1,\ldots, m}|w_i(x,t)|$
	\begin{equation*}
	\begin{aligned}
		\sup_{\Omega\times(\tau,\tau+2)}|\nx w_i(x,t)| &\leq C W^{1/2}\sup_{\Omega\times(\tau,\tau+2)}\left|\varphi'_\tau u_i(x,t) + \varphi_\tau f_i(u)\right|^{1/2}\\
		&\leq CW^{1/2}\left(1+U^{\frac{2+\varepsilon}{2}}\right)
		\leq CU^{1/2}\left(1+U^{\frac{2+\varepsilon}{2}}\right)
		\leq C\left(1+ U^{\frac{3+\varepsilon}{2}}\right)
	\end{aligned}
	\end{equation*}
	since $w_i(x,t) = \varphi_\tau u_i(x,t)$.
\end{proof}
\begin{lemma}\label{lem5}
	Define
	\begin{equation*}
		v_i(x,t) = \int_{\tau}^{t}u_i(x,s)ds
	\end{equation*}
	for $(x,t)\in\Omega\times(\tau,\tau+2)$. Then
	\begin{equation*}
		\sup_{\Omega\times(\tau,\tau+2)}|\nx v_i(x,t)| \leq C\left(1 + U^{\frac{1+\varepsilon}{2}} \right)
	\end{equation*}
	where the constant $C$ is independent of $\tau$.
\end{lemma}
\begin{proof}
	From the definition we see that $v_i$ solves the equation
	\begin{equation*}
		\begin{cases}
			\partial_t v_i(x,t) - d_i\Delta v_i(x,t) = \int_{\tau}^{t}f_i(u(x,s))ds, &x\in\Omega, \; t\in (\tau,\tau+2),\\
			\nx v_i \cdot \nu = 0, &x\in\partial\Omega, \; t\in (\tau,\tau+2),\\
			v_i(x,\tau) = 0, &x\in\Omega.
		\end{cases}
	\end{equation*}
	Using the condition \eqref{A3} and Lemma \ref{lem2} we have
	\begin{align*}
		\sup_{\Omega\times(\tau,\tau+2)}\left|\int_\tau^tf_i(u(x,s))ds\right| &\leq C\sup_{\Omega\times(\tau,\tau+2)}\int_{\tau}^{\tau+2}\left[1 + \sum_{i=1}^mu_i(x,s)\right]\left[1 + \sum_{i=1}^mu_i(x,s)\right]^{1+\varepsilon}ds\\
		&\leq C(1+U)^{1+\varepsilon}\sup_{\Omega\times(\tau,\tau+2)}\int_{\tau}^{\tau+2}\left[1+\sum_{i=1}^mu_i(x,s) \right]ds\\
		&\leq C(1+U^{1+\varepsilon})
	\end{align*}
	with a constant $C$ is independent of $\tau$. We use this estimate and Lemma \ref{key_lem} to the equation of $v_i$ to have
	\begin{equation*}
	\begin{aligned}
		\sup_{\Omega\times(\tau,\tau+2)}|\nx v_i(x,t)| &\leq C(\sup_{\Omega\times(\tau,\tau+2)}|v_i(x,t)|)^{1/2}\left(\sup_{\Omega\times(\tau,\tau+2)}\left|\int_\tau^tf_i(u(x,s))ds\right| \right)^{1/2}\\
		&\leq C\left(\sup_{\Omega\times(\tau,\tau+2)}\left|\int_{\tau}^tu(x,s)ds \right|\right)^{1/2}\left(1+U^{1+\varepsilon} \right)^{1/2}\\
		&\leq C\left(1+U^{\frac{1+\varepsilon}{2}} \right)
	\end{aligned}
	\end{equation*}
	with a constant $C$ being independent of $\tau$, which finishes the proof of Lemma \ref{lem5}.
\end{proof}
\begin{lemma}\label{lem6}
	Let $v$ defined in \eqref{def_v}. Then there exists a constant $C$ independent of $\tau$ such that
	\begin{equation*}
		\sup_{\Omega\times(\tau,\tau+2)}|\Delta v(x,t)| \leq C\left(1+U^{\frac{3+\varepsilon}{4} + \frac{1-\delta}{2(2-\delta)}} \right)
	\end{equation*}	
	where $\delta$ is in Lemma \ref{lem3}.
\end{lemma}
\begin{proof}
	We apply Lemma \ref{key_lem1} to equation \eqref{eq_v1} to have
	\begin{equation}\label{e13}
		\sup_{\Omega\times(\tau,\tau+2)}|\Delta v(x,t)| \leq C(\sup_{\Omega\times(\tau,\tau+2)}[|v| + |\nx v|])^{1/2}(\sup_{\Omega\times(\tau,\tau+2)}[|\psi| + |\nx \psi|])^{1/2}.
	\end{equation}
	From the definition \eqref{def_v} and Lemma \ref{lem2}, 
	\begin{equation}\label{e14}
		\sup_{\Omega\times(\tau,\tau+2)}|v| \leq C.
	\end{equation}
	From \eqref{psi}, the definition $w_i(x,t) = \varphi_\tau u_i(x,t)$ and Lemma \ref{lem2} we have
	\begin{equation}\label{e15}
		\sup_{\Omega\times(\tau,\tau+2)}|\psi| \leq C(1+U).
	\end{equation}
	The estimate for gradient of $\psi$,
	\begin{equation*}
		\sup_{\Omega\times(\tau,\tau+2)}|\nx \psi| \leq C\left(1+U^{\frac{3+\varepsilon}{2}}\right)
	\end{equation*}
	is implied from Lemmas \ref{lem4} and \ref{lem5}. Putting \eqref{e13}--\eqref{e15} and Lemma \ref{lem3_1} into \eqref{e13} we have
	\begin{equation*}
		\sup_{\Omega\times(\tau,\tau+2)}|\Delta v(x,t)| \leq C\left(1+U^{\frac{1-\delta}{2-\delta}}\right)^{1/2}\left(1+U^{\frac{3+\varepsilon}{2}} \right)^{1/2} \leq C\left(1 + U^{\frac{3+\varepsilon}{4} + \frac{1-\delta}{2(2-\delta)}}\right)
	\end{equation*}
	with a constant $C$ independent of $\tau$, which completes the proof of Lemma \eqref{lem6}.
\end{proof}

We are now ready to prove the main theorem.
\begin{proof}[Proof of Theorem \ref{thm:main}]
	The first equation in \eqref{eq_v} can be rewritten as
	\begin{equation*}
		\sum_{i=1}^mw_i(x,t) = \Delta v(x,t) + \varphi'_\tau \int_{\tau}^{t}u_i(x,s)ds
	\end{equation*}
	for all $x\in\Omega$ and $t\in (\tau,\tau+2)$. Now we apply Lemmatax \ref{lem6} and \ref{lem2} to get
	\begin{equation}\label{e16}
		\sup_{\Omega\times(\tau,\tau+2)}\sup_{i=1,\ldots, m}|w_i(x,t)| \leq C\left(1 + U^{\frac{3+\varepsilon}{4} + \frac{1-\delta}{2(2-\delta)}} \right),
	\end{equation}
	with a constant $C$ independent of $\tau$, recalling that 
	\begin{equation*}
		U = \sup_{\Omega\times(\tau,\tau+2)}\sup_{i=1,\ldots, m}|u_i(x,t)|.
	\end{equation*}
	We consider $\tau\in \mathbb N$ such that	
	\begin{equation}\label{e17}
		\sup_{\Omega\times(\tau,\tau+1)}\sup_{i=1,\ldots, m}|u_i(x,t)| \leq \sup_{\Omega\times(\tau+1,\tau+2)}\sup_{i=1,\ldots, m}|u_i(x,t)|.
	\end{equation}
	Now using the fact that $w_i(x,t) = \varphi_\tau u_i(x,t) = u_i(x,t)$ for all $t\in (\tau+1,\tau+2)$ we obtain from \eqref{e16} and \eqref{e17} that
	\begin{equation*}
		U \leq 2\sup_{\Omega\times(\tau+1,\tau+2)}\sup_{i=1,\ldots, m}|u_i(x,t)| \leq 2C\left(1 + U^{\frac{3+\varepsilon}{4} + \frac{1-\delta}{2(2-\delta)}}\right).
	\end{equation*}
	Because $\delta \in (0,1)$, we can choose $\varepsilon>0$ small enough such that
	\begin{equation*}
		\frac{3+\varepsilon}{4} + \frac{1-\delta}{2(2-\delta)} < 1
	\end{equation*}
	and thus by Young's inequality we get
	\begin{equation*}
		U = \sup_{\Omega\times(\tau,\tau+2)}\sup_{i=1,\ldots, m}|u_i(x,t)|\leq C
	\end{equation*}
	for all $\tau$ such that \eqref{e17} holds. Thanks to Lemma \ref{elementary}, this estimate is also true for all $\tau\in \mathbb N$, with a constant $C$ {\it independent} of $\tau$. This proves the uniform bounds
	\begin{equation*}
		\sup_{\Omega\times[0,\infty)}\sup_{i=1,\ldots, m}|u_i(x,t)| \leq C
	\end{equation*}
	and the proof of Theorem \ref{thm:main} is complete.
\end{proof}
\section{Application to skew-symmetric Lotka-Volterra systems}\label{application}
In this section, we give an application our main result to skew-symmetric Lotka-Volterra systems in three dimensions
\begin{equation}\label{LV}
	\begin{cases}
		\partial_t u_i - d_i\Delta u_i  = -\tau_i u_i + u_i\sum_{j=1}^m a_{ij}u_j, &x\in\Omega, \; t>0, \; i=1,\ldots, m,\\
		\nx u_i\cdot \nu = 0, &x\in\partial\Omega, \; t>0, \; i=1,\ldots, m,\\
		u_i(x,0) = u_{i0}(x), &x\in\Omega, \;i=1,\ldots, m,
	\end{cases}
\end{equation}
where the matrix $A = (a_{ij})_{i,j=1,\ldots, m}$ satisfies 
\begin{equation*}
	 A + A^{\top} = 0.
\end{equation*}
This system was studied previously in one and two dimensions, see e.g. \cite{SY13, PSY17}, in which the solution to \eqref{LV} was proved to be bounded uniformly in time if $\tau_i \geq 0$ for all $i=1, \ldots, m$. Moreover, the trajectory was shown to be relatively compact in $(C(\overline{\Omega}))^m$. In higher dimensions, it was shown recently in \cite{FMT19} that the classical solution exists globally in any dimension with possible polynomial growth in time. In addition, if $\tau_i > 0$ for all $i=1,\ldots, m$, then the solution converges exponentially to zero as $t\to \infty$. 

Up to our best knowledge, the uniform-in-time boundedness of the solution to \eqref{LV} when $n\geq 3$ and $\tau_i \geq 0$ for all $i=1,\ldots, m$, is completely unknown. 
Our main result in Theorem \ref{thm:main} provides the answer affirmatively in all dimensions. Moreover, we can relax the equal sign in $A + A^\top = 0$ to the smaller or equal sign.
\begin{theorem}
	Let $\Omega\subset\mathbb R^n$, $n\geq 3$, be a bounded domain with smooth boundary. Assume that $\tau_i \geq 0$ for all $i=1,\ldots, m$ and componentwise
	\begin{equation*}
		A + A^\top \leq 0.
	\end{equation*}
	Then, for any bounded, nonnegative initial data, the system \eqref{LV} has a unique bounded classical solution which is also bounded uniformly in time, i.e.
	\begin{equation*}
		\sup_{t\geq 0}\sup_{i=1,\ldots, m}\|u_i(t)\|_{\infty,\Omega} < +\infty.
	\end{equation*}
	Moreover, the trajectory is relatively compact in $(C(\overline{\Omega}))^m$.
\end{theorem}
\begin{proof}
	The global existence and uniform-in-time bound of the classical solution follows directly from Theorem \ref{thm:main}. Now we can apply Lemma \ref{key_lem} to each equation of $u_i$ to have
	\begin{equation*}
		\sup_{\Omega\times(0,T)}|\nx u_i(x,t)| \leq C\|u_{i0}\|_{C^1(\Omega)} + B\sup_{\Omega\times(0,T)}|u_i(x,t)|^{1/2}\sup_{\Omega\times(0,T)}|f_i(u(x,t))|^{1/2} \leq C
	\end{equation*}
	for $C$ is independent of $T$. Hence the trajectory $\{u(t)\}_{t\geq 0}$ is bounded uniformly in $C^1(\overline{\Omega})^m$. The relative compactness in $C(\overline{\Omega})^m$ therefore follows from the Arzel\`{a}-Ascoli theorem.
\end{proof}

\par{\bf Acknowledgements:} 
This work is supported by the International Training Program IGDK 1754 and NAWI Graz.

\end{document}